\documentclass[a4paper,12pt]{amsart}

\usepackage{amssymb}
\usepackage{latexsym}
\usepackage{amsmath}
\usepackage{amscd}
\usepackage{graphicx}
\usepackage{palatino}
\usepackage{xcolor}
\usepackage{amsrefs}

\title[Flat $S^1$-bundle]
{
Remarks on 
flat $S^1$-bundles,
$C^\infty$ vs $C^\omega$ 
}

\author{Teruaki Kitano}
\address{Department of information Systems Science, 
Faculty of Science and Engineering,
Soka University, 
Tangi-machi 1-236, 
Hachioji, Tokyo, 192-8577, Japan}
\email{kitano@soka.ac.jp}

\author{Yoshihiko Mitsumatsu}
\address{Department of Mathematics, Chuo University, 
1-13-27 Kasuga Bunkyo-ku, 
Tokyo, 112-8551, Japan}
\email{yoshi@math.chuo-u.ac.jp}

\author{Shigeyuki Morita}
\address{Graduate School of Mathematical Sciences, 
The University of Tokyo, 
3-8-1 Komaba, 
Meguro-ku, Tokyo, 153-8914, Japan}
\email{morita@ms.u-tokyo.ac.jp}

\subjclass[2010]{Primary~55R40, 57R32
}
\keywords{flat $S^1$-bundle, Euler class,
Haefliger $\Gamma$-structure, 
Mather-Thurston theory, 
Borel construction
}

\newtheorem{thm}{Theorem}[section]
\newtheorem{prop}[thm]{Proposition}

\newtheorem{cor}[thm]{Corollary}
\theoremstyle{definition}
\newtheorem{definition}[thm]{Definition}
\newtheorem{example}[thm]{Example}
\newtheorem{remark}[thm]{Remark}
\newtheorem{problem}[thm]{Problem}

\begin{document}

\newcommand{\Mg}{\mathcal{M}_g}
\newcommand{\Mgp}{\mathcal{M}_{g,\ast}}
\newcommand{\Mgb}{\mathcal{M}_{g,1}}

\newcommand{\hg}{\mathfrak{h}_{g,1}}
\newcommand{\ag}{\mathfrak{a}_g}
\newcommand{\Ln}{\mathcal{L}_n}

\newcommand{\Sg}{\Sigma_g}
\newcommand{\Sgb}{\Sigma_{g,1}}
\newcommand{\la}{\lambda}

\newcommand{\Symp}[1]{Sp(2g,\mathbb{#1})}
\newcommand{\symp}[1]{\mathfrak{sp}(2g,\mathbb{#1})}
\newcommand{\gl}[1]{\mathfrak{gl}(n,\mathbb{#1})}

\newcommand{\At}[1]{\mathcal{A}_{#1}^t (H)}
\newcommand{\Hq}{H_{\mathbb{Q}}}

\newcommand{\Ker}{\mathop{\mathrm{Ker}}\nolimits}
\newcommand{\Hom}{\mathop{\mathrm{Hom}}\nolimits}
\renewcommand{\Im}{\mathop{\mathrm{Im}}\nolimits}

\newcommand{\Der}{\mathop{\mathrm{Der}}\nolimits}
\newcommand{\Out}{\mathop{\mathrm{Out}}\nolimits}
\newcommand{\Aut}{\mathop{\mathrm{Aut}}\nolimits}
\newcommand{\Q}{\mathbb{Q}}
\newcommand{\Z}{\mathbb{Z}}
\newcommand{\R}{\mathbb{R}}
\newcommand{\C}{\mathbb{C}}

\begin{abstract}
We describe 
low dimensional homology groups of 
$\mathrm{Diff}^\delta_+S^1$ 
in terms of Haefliger's classifying space $B\overline{\Gamma}_1$
by 
applying a theorem of Thurston.
Then we consider the question whether some power of the rational
Euler class vanishes for real analytic flat $S^1$-bundles.
We show that if it occurs, then
the homology
group 
of $\mathrm{Diff}_+^{\omega,\delta} S^1$
should contain
two kinds of
many torsion classes which vanish in $\mathrm{Diff}^\delta_+S^1$. 
This is an informal note on our discussions about the above question (see Remark 1.17).
\end{abstract}

\renewcommand\baselinestretch{1.1}
\setlength{\baselineskip}{16pt}

\newcounter{fig}
\setcounter{fig}{0}

\maketitle

\vspace{1mm}

\section{Results}

Let $\mathrm{Diff}_+S^1$ be the orientation preserving $C^\infty$ diffeomorphism group
of the circle with the smooth topology 
and let $\mathrm{Diff}^\delta_+S^1$ denote the same group 
equipped with the {\it discrete} topology. 
Then there is a fibration
$$
B\overline{\mathrm{Diff}}_+ S^1\rightarrow B\mathrm{Diff}^\delta_+S^1
\rightarrow B\mathrm{Diff}_+S^1
$$
where $B\mathrm{Diff}^\delta_+S^1$ is the classifying space for 
flat $S^1$-bundles
while $B\overline{\mathrm{Diff}}_+ S^1$ is the classifying space for 
flat $S^1$-products. Since $\mathrm{Diff}_+S^1$ is homotopy equivalent to $\mathrm{SO}(2)$,
if we denote by $\widetilde{\mathrm{Diff}}_+S^1$ the universal covering group of
$\mathrm{Diff}_+S^1$, then $B\widetilde{\mathrm{Diff}}^\delta_+S^1$ can serve as
$B\overline{\mathrm{Diff}}_+ S^1$.
As is well known,
there are natural identification
$$
\widetilde{\mathrm{Diff}}_+S^1\cong \{f\in \mathrm{Diff}_+^\infty \R; Tf=fT\}\quad
\text{where $T(x)=x+1\ (x\in \R)$}
$$
and a cental extension
\begin{equation}
0\rightarrow \Z\rightarrow\widetilde{\mathrm{Diff}}_+S^1\overset{p}{\rightarrow}
\mathrm{Diff}_+S^1\rightarrow 1.
\label{eq:ce}
\end{equation}

Now let us recall a theorem of Thurston which says that 
$B\overline{\mathrm{Diff}}_+ S^1$ and hence $B\widetilde{\mathrm{Diff}}^\delta_+ S^1$
is homologically equivalent to the free loop space $\wedge B\overline{\Gamma}_1$
of Haefliger's classifying space $B\overline{\Gamma}_1$ (\cite{MR0285027,MR100269}).

\begin{thm}[Thurston \cite{MR339267}]
Let 
$h:
B\widetilde{\mathrm{Diff}}^\delta_+ S^1\times S^1 \rightarrow B\overline{\Gamma}_1
$
be the classifying map for the flat $S^1$-product over
$B\widetilde{\mathrm{Diff}}^\delta_+ S^1$.
Then its adjoint mapping
$$
H: 
B\widetilde{\mathrm{Diff}}^\delta_+ S^1\rightarrow \wedge B\overline{\Gamma}_1
$$
induces isomorphism on homology.
\label{th:Thurston}
\end{thm}

By making use of this theorem, we obtain the following results.

\begin{thm}
$\mathrm{(i)}\ $
There exist isomorphisms:
\begin{align*}
H_2(B\mathrm{Diff}_+^\delta S^1;\Z)&\cong H_3(B\overline{\Gamma}_1;\Z)\oplus\Z \quad
(\text{canonical direct sum}),\\
H_2(B\widetilde{\mathrm{Diff}}_+^\delta S^1;\Z)&\cong H_3(B\overline{\Gamma}_1;\Z).\
\end{align*}

$\mathrm{(ii)}\ $
There exist isomorphisms:
\begin{align*}
H_3(B\mathrm{Diff}_+^\delta S^1;\Z)&\cong H_3(\Omega B\overline{\Gamma}_1;\Z),\\
H_3(B\widetilde{\mathrm{Diff}}_+^\delta S^1;\Z)&\cong H_3(B\overline{\Gamma}_1;\Z)\oplus
H_3(\Omega B\overline{\Gamma}_1;\Z).
\end{align*}

$\mathrm{(iii)}\ $
If we denote by 
$\mu: H_2(B\mathrm{Diff}_+^\delta S^1;\Z)\rightarrow H_3(B\widetilde{\mathrm{Diff}}_+^\delta S^1;\Z)$
a part of the Gysin exact sequence associated with the central extension \eqref{eq:ce},
then 
it is given as follows.
\begin{align*}
H_2(B\mathrm{Diff}_+^\delta S^1;\Z)\cong &H_3(B\overline{\Gamma}_1;\Z)\oplus\Z
\ni (\sigma,n)\\
&\overset{\mu}{\longmapsto} (\sigma,0)\in H_3(B\widetilde{\mathrm{Diff}}_+^\delta S^1;\Z)\cong
H_3(B\overline{\Gamma}_1;\Z)\oplus H_3(\Omega B\overline{\Gamma}_1;\Z).
\end{align*}
The generator $1\in \Z$ of the canonical summand $\Z\subset H_2(B\mathrm{Diff}_+^\delta S^1;\Z)$ 
in $\mathrm{(i)}$ is 
characterized by the two conditions $\mu(1)=0$ and $\chi(1)=1$,
where $\chi\in H^2(B\mathrm{Diff}_+^\delta S^1;\Z)$ denotes the Euler class.
\label{th:main}
\end{thm}

Recall here that Herman \cite{herman1, herman2} proved 
$H_1(B\mathrm{Diff}_+^\delta S^1;\Z)=H_1(B\mathrm{Diff}_+^{\omega,\delta}S^1;\Z)=0$. 

\begin{example}
We describe an element belonging to the canonical direct summand $\Z\subset H_2(B\mathrm{Diff}_+^\delta S^1;\Z)$.
Here we use Thurston's original idea by which he proved linear independence of the two classes $\chi,\alpha\in H^2(B\mathrm{Diff}_+^\delta S^1;\R)$ 
where $\alpha$ denotes the Godbillon-Vey class integrated along the fiber.

Let 
$$
\rho:\pi_1(\Sigma_2)\rightarrow \mathrm{PSL}(2,\R)\rightarrow \mathrm{Diff}_+^{\omega,\delta} S^1
$$ 
be a Fuchsian representation, corresponding to a hyperbolic structure
on a closed surface $\Sigma_2$ of genus $2$, followed by a natural embedding 
$\mathrm{PSL}(2,\R)\subset \mathrm{Diff}_+^{\omega,\delta} S^1$. 
It is known that this representation lifts to $\mathrm{SL}(2,\R)$
(see \cite{MR0440554}) so that we have
$$
\tilde{\rho}:\pi_1(\Sigma_2)\rightarrow \mathrm{SL}(2,\R)\rightarrow \mathrm{Diff}_+^{\omega,\delta} S^1
$$
where the second homomorphism is induced by the natural action of $\mathrm{SL}(2,\R)$
on the set of oriented directions from the origin of $\R^2$.
Then $\chi(\rho_*([\Sigma_2]))=-2$ while $\chi(\tilde{\rho}_*([\Sigma_2]))=-1$. Now set 
$\sigma=\tilde{\rho}_*([\Sigma_2])-2\rho_*([\Sigma_2])\in H_2(B\mathrm{Diff}_+^\delta S^1;\Z)$ so that $\chi(\sigma)=3$.
We show that $\mu(\sigma)=0$ which implies that 
$\sigma$ represents
$3$ of the canonical direct summand $\Z\subset H_2(B\mathrm{Diff}_+^\delta S^1;\Z)$.

We denote by $S^1(\rho)$ (resp. $S^1(\tilde\rho)$) the total space of the flat
$S^1$-bundle induced by $\rho$ (resp. $\tilde\rho$). Then there exists a
fiberwise $2$-fold covering map $S^1(\tilde\rho)\rightarrow S^1(\rho)$. Therefore
the classifying map for the codimension one foliations
on these total spaces factors as 
$$
S^1(\tilde\rho) \xrightarrow[\text{$2$-fold cover}]{\text{fiberwise}} S^1(\rho)\longrightarrow B\overline{\Gamma}_1.
$$
The images in $H_3(B\overline{\Gamma}_1;\Z)$ of $[S^1(\tilde\rho)], [S^1(\rho)]$ under the above map
represent the first summands of $\mu(\tilde{\rho}_*[\Sigma_2]),\mu(\rho_*[\Sigma_2])$ and the 
first one is twice the second one.
Therefore
Theorem \ref{th:main} $\mathrm{(ii), (iii)}$ implies that $\mu(\sigma)=0\in H_3(B\widetilde{\mathrm{Diff}}_+^\delta S^1;\Z)$
as required. 

Instead of $\mathrm{PSL}(2,\R)$, we may use the group $\mathrm{GL}^+(2,\Z[\frac{1}{2}])$ using a result of Milnor
\cite{MR95518}. 
Also it is a very important question whether 
$\mu(\sigma)=0$ holds already in $H_3(B\widetilde{\mathrm{Diff}}_+^{\omega,\delta} S^1;\Z)$ or not.
This is because if it holds, then we can conclude that $\chi^2$ does not vanish rationally for real analytic flat $S^1$-bundles.
\label{ex:canonical}
\end{example}

Let
$
\varphi_k:\widetilde{\mathrm{Diff}}_+^\delta S^1\rightarrow\widetilde{\mathrm{Diff}}_+^\delta S^1\ (k=2,3,\cdots)
$
be the endomorphism defined by
$$
\varphi_k(f)(x)=\frac{1}{k} f(kx)\quad (f\in \widetilde{\mathrm{Diff}}_+^\delta S^1).
$$
In the situation of the above Example \ref{ex:canonical}, the endomorphism $\varphi_2$ appears in the following commutative diagram.
$$
\begin{CD}
\pi_1(S^1(\tilde{\rho}))
 @>>> \widetilde{\mathrm{Diff}}_+^\delta S^1
\\
@V{\cap}V{\text{index $2$}}V  @AA{\varphi_2}A\\
\pi_1(S^1(\rho)) @>>>\widetilde{\mathrm{Diff}}_+^\delta S^1.
\end{CD}
$$

\begin{thm}
For any element $\sigma\in H_2(B\mathrm{Diff}_+^\delta S^1;\Z)$,
and for any $k$, we have
$$
(\varphi_k)_* (\mu(\sigma))=\mu(\sigma)\in H_3(B\widetilde{\mathrm{Diff}}_+^\delta S^1;\Z).
$$
\label{th:phi}
\end{thm}

\begin{problem}
Study 
the above equality in the real analytic case. Namely, 
for a given element $\sigma\in H_2(B\mathrm{Diff}_+^{\omega,\delta} S^1;\Z)$
and $k$, determine whether the following equality holds or not.
$$
(\varphi_k)_* (\mu(\sigma))=\mu(\sigma)\in H_3(B\widetilde{\mathrm{Diff}}_+^{\omega,\delta} S^1;\Z)
$$
\end{problem}

This is related to the question of non-triviality of $\chi^2\in H^4(B\mathrm{Diff}_+^{\omega,\delta} S^1;\Q)$
because of the following result.

By the way, even in the smooth case, 
it is an open problem to construct explicit $4$-cycles in $H_4(B\mathrm{Diff}_+^{\delta} S^1;\Z)$ 
with non-vanishing $\chi^2$. 

\begin{prop}
If the above problem will be affirmatively solved for one particular element 
$\sigma\in H_2(B\mathrm{Diff}_+^{\omega,\delta} S^1;\Z)$ with $\chi(\sigma)\not=0$ and 
one particular $k$,
then we have
$$
\chi^2\not= 0\in H^4(B\mathrm{Diff}_+^{\omega,\delta} S^1;\Q).
$$
\label{prop:chi}
\end{prop}

Next, we consider what can be said about
the homology of $B\mathrm{Diff}_+^{\omega,\delta} S^1$
assuming that 
$\chi^2=0\in H^4(B\mathrm{Diff}_+^{\omega,\delta} S^1;\Q)$
(or more generally $\chi^k=0$ for some $k\geq 2$).
We show that there will arise rather strange integral homology classes
of $B\mathrm{Diff}_+^{\omega,\delta} S^1$.
By Proposition \ref{prop:chi}, the above assumption is equivalent to the
following condition:
\begin{align*}
(\varphi_k)_* (\mu(\sigma))&\not=\mu(\sigma)\in H_3(B\widetilde{\mathrm{Diff}}_+^{\omega,\delta} S^1;\Z)\\
&\text{for any $\sigma\in H_2(B\mathrm{Diff}_+^{\omega,\delta} S^1;\Z)$ with $\chi(\sigma)\not=0$
and $k\geq 2$}.
\end{align*}

\begin{thm}
Assume that $\chi^2=0\in H^4(B\mathrm{Diff}_+^{\omega,\delta} S^1;\Q)$. Then
the quotient group
$$
H_{3}(B\widetilde{\mathrm{Diff}}_+^{\omega,\delta}S^1;\Z)
/\mu(p_\ast(H_{2}(B\widetilde{\mathrm{Diff}}_+^{\omega,\delta}S^1;\Z)))
$$
contains a group $P\ (\supset \Z)$ which
admits a surjective homomorphism onto $\Q$. This group $P$ vanishes in 
$H_3(B\widetilde{\mathrm{Diff}}_+^{\delta}S^1;\Z)$
and there is a subgroup
$$
P/\Z\subset H_{3}(B\mathrm{Diff}_+^{\omega,\delta}S^1;\Z)
$$
which admits a surjective homomorphism onto $\Q/\Z$.
\label{th:Q}
\end{thm}

\begin{remark}
In general, we can prove the following statement.

Assume that $\chi^{k-1}\not=0\in H^{2k-2}(B\mathrm{Diff}_+^{\omega,\delta} S^1;\Q)$
and $\chi^{k}=0\in H^{2k}(B\mathrm{Diff}_+^{\omega,\delta} S^1;\Q)$
for some $k\geq 2$. Then we can conclude that the quotient group
$$
H_{2k-1}(B\widetilde{\mathrm{Diff}}_+^{\omega,\delta}S^1;\Z)
/\mu(p_\ast(H_{2k-2}(B\widetilde{\mathrm{Diff}}_+^{\omega,\delta}S^1;\Z)))
$$
contains a group $P\ (\supset \Z)$
which admits a surjective homomorphism onto $\Q$. This group $P$
vanishes in $H_{2k-1}(B\widetilde{\mathrm{Diff}}_+^{\delta}S^1;\Z)$
and there is a subgroup
$$
P/\Z\subset H_{2k-1}(B\mathrm{Diff}_+^{\omega,\delta}S^1;\Z)
$$
which admits a surjective homomorphism onto $\Q/\Z$.
\label{rem:Q}
\end{remark}

\vspace{2mm}
\begin{remark}
As shown in Dupont-Sah \cite{MR662760, MR722374} and Parry-Sah \cite{MR722372},
a similar phenomenon as in Theorem \ref{th:Q} 
occurs at the level of $\mathrm{SL}^\delta(2,\R)\subset \mathrm{Diff}_+^{\delta}S^1$.
Recall here that $\chi^2=0\in H^4(B\mathrm{SL}^\delta(2,\R);\Q)$).
More precisely, they proved that the following part of the Gysin exact sequence
for the central extension $
0\rightarrow \Z\rightarrow \widetilde{\mathrm{SL}}\rightarrow 
\mathrm{SL}\rightarrow 1
$
(where $\mathrm{SL}$ denotes $\mathrm{SL}(2,\R)$ for short)
contains a sub-exact sequence shown in the second row
below:
\begin{equation}
\begin{CD}
 H_2(B\mathrm{SL}^{\delta};\Z) @>{\mu}>> H_3(B\widetilde{\mathrm{SL}}^{\delta};\Z)
 @>>>H_3(B\mathrm{SL}^{\delta};\Z)  @>>>H_1(\mathrm{SL}^{\delta};\Z)=0\\
 @A{\cup}AA  @A{\cup}AA @A{\cup}AA @|\\
  \Z @>{\mu}>> \Q
  @>>> \Q/\Z@>>> 0.
\end{CD}
\label{eq:slQ}
\end{equation}
However, there is also a considerable difference between the two cases. The $\Q$-factor in
\eqref{eq:slQ} survives in $H_3(B\widetilde{\mathrm{Diff}}^\delta_+S^1;\Z)$
while the $\Q$-factor
in Theorem \ref{th:Q} vanishes there. This is because, the former one is
detected by the $\beta$ class $\in H^3(B\widetilde{\mathrm{Diff}}^\delta_+S^1;\R)$ 
(= Godbillon-Vey class) pulled back to $H^3(B\widetilde{\mathrm{SL}}^{\delta};\R)$
while the latter is not. On the other hand, 
the $\Q/\Z$-factor in \eqref{eq:slQ} is described as
$
H_3(\mathrm{SO}(2)_{\text{tor}};\Z)\cong H_3(\Q/\Z;\Z)\cong
\Q/\Z\
$
where $\mathrm{SO}(2)\subset \mathrm{SL}$ is the subgroup consisting of 
rotations and $\mathrm{SO}(2)_{\text{tor}}$ denotes its torsion subgroup.
Dupont and Sah \cite{MR662760, MR722374} proved that it is detected by
the Cheeger-Chern-Simons class $\hat{c}_2
$.
\label{rem:dsp}
\end{remark}

Recall here that 
$\mathrm{SO}(2)_{\text{tor}}\cong\Q/\Z
=\underset{{\underset{n}{\longrightarrow}}}{\lim} \frac{1}{n}\Z/\Z$ 
where $\frac{1}{n}\Z/\Z$ is the subgroup of  $\mathrm{SO}(2)_{\text{tor}}$ 
and also recall 
$H_{2k-1}(\mathrm{SO}(2)_{\text{tor}};\Z)\cong\Q/\Z$ 
where $\frac{1}{n}\Z/\Z=\Z/n\Z$ is realized 
by the natural flat $S^1$-bundle over the lens space $L_n^{2k-1}$. 

Motivated by the above result, we consider how the homology group
$H_*(\mathrm{SO}(2)_{\text{tor}};\Z)$ of the subgroup
$\mathrm{SO}(2)_{\text{tor}}\subset \mathrm{Diff}^\delta_+S^1$ consisting of rational rotations of $S^1$
will survive (or vanish) in the homology of various subgroups of $\mathrm{Diff}^\delta_+S^1$ containing 
$\mathrm{SO}(2)$, in particular $\mathrm{Diff}^\delta_+S^1$ itself and $\mathrm{Diff}^{\omega,\delta}_+S^1$.
As for the former $C^\infty$ case, we obtain the following vanishing result by using 
Thurston's Theorem \ref{th:Thurston}.

\begin{thm}
Let $\mathrm{SO}(2)_{\text{tor}}\subset \mathrm{Diff}^\delta_+S^1$ be the subgroup consisting of all the
rational rotations. Then
the homomorphisms
$$
H_{2k-1}(\mathrm{SO}(2)_{\text{tor}};\Z)\cong
\Q/\Z\longrightarrow H_{2k-1}(B\mathrm{Diff}^\delta_+S^1;\Z)\quad (k=2,3,\cdots)
\label{eq:3}
$$
are all trivial.
\label{th:other}
\end{thm}

\begin{remark}
$\mathrm{(i)}\ $
By considering $\mathrm{SO}(2)_{\text{tor}}$,  
Nariman \cite{nariman} proved that the homomorphism 
\[
H^*(B\mathrm{Diff}_+ S^1;\Z)\cong H^*(\C P^{\infty};\Z) \rightarrow
H^*(B\mathrm{Diff}_+^{\omega,\delta} S^1;\Z)\]
is injective, by showing the fact that
\[
H^*(B\mathrm{Diff}_+^{\omega,\delta} S^1;\Z)  \rightarrow H^*(\mathrm{SO}(2)_{\text{tor}};\Z)\rightarrow H^*(\Z/n\Z;\Z)
\]
is surjective for any $\Z/n\Z \subset \mathrm{SO}(2)_{\text{tor}}$. 

In contract with this, we showed that the homology map
\[
H_*(\mathrm{SO}(2)_{\text{tor}};\Z) \rightarrow H_*(B\mathrm{Diff}_+^{\delta} S^1;\Z)
\]
is trivial. 
The problem is whether this still holds in $H_*(B\mathrm{Diff}_+^{\omega,\delta} S^1;\Z)$ or not. 

$\mathrm{(ii)}\ $
We also remark that if $B\overline{\Gamma}_1$ were an Eilenberg-MacLane 
space $K(\R,3)$,
then we can compute the homology group $H_{*}(B\mathrm{Diff}^\delta_+S^1;\Z)$
explicitly, again by using Thurston's Theorem \ref{th:Thurston} (see Remark \ref{rem:kr3}).
In particular, we see that it
has no torsion,
although this assumption would be too naive at present.
\end{remark}

On the other hand, 
we obtain the following
facts which may give another method of proving the non-triviality of 
the rational $\chi^2$
(and more generally the problem of determining whether the rational
$\chi^k=0$ for some $k$ or not) in the
real analytic case.
\begin{prop}
If the homomorphism
$$
H_3(\mathrm{SO}(2)_{\text{tor}};\Z)\cong
\Q/\Z\longrightarrow H_3(B\mathrm{Diff}^{\omega,\delta}_+S^1;\Z)
$$
is not injective, then $\chi^2\not=0 \in H^4(B\mathrm{Diff}^{\omega,\delta}_+S^1;\Q)$. 
\end{prop}

In fact, this is a particular case of the following more general facts.

\begin{thm}
$\mathrm{(i)}$\ Let $\Gamma\subset \mathrm{Diff}^\delta_+S^1$ be any subgroup containing 
$\mathrm{SO}(2)$. Assume that the homomorphism
$$
i_*:H_{2k-1}(\mathrm{SO}(2)_{\text{tor}};\Z)\cong
\Q/\Z\longrightarrow H_{2k-1}(\Gamma;\Z)
$$
is injective (resp. non-trivial) for some $k$, where $i:\mathrm{SO}(2)_{\text{tor}}\subset \Gamma$
denotes the inclusion. Then the homomorphisms
$$
i_*: H_{2l-1}(\mathrm{SO}(2)_{\text{tor}};\Z)\cong
\Q/\Z\longrightarrow H_{2l-1}(\Gamma;\Z)
$$
are injective (resp. non-trivial) for all $l\geq k$.

$\mathrm{(ii)}$\
Assume that $\chi^{k}=0\in H^{2k}(\Gamma;\Q)$. Then the homomorphisms
$$
i_*: H_{2l-1}(\mathrm{SO}(2)_{\text{tor}};\Z)\cong
\Q/\Z\longrightarrow H_{2l-1}(\Gamma;\Z)
$$
are injective for all $l\geq k$.
\label{th:hatchi}
\end{thm}

\begin{remark}
If we combine Theorem \ref{th:other} and Theorem \ref{th:hatchi} (ii) for the case 
$\Gamma=\mathrm{Diff}^{\delta}_+S^1$, we obtain yet another proof of the fact that
all the powers of the rational Euler class are non-trivial in $H^*(B\mathrm{Diff}^{\delta}_+S^1;\Q)$
(see \cite{MR759480}\cite{gs}\cite{nariman} for former proofs).
\end{remark}

\begin{example}
$\mathrm{(i)}$\ 
The results of Dupont, Sah and Parry mentioned in Remark \ref{rem:dsp}
together with Theorem \ref{th:hatchi}\ $\mathrm{(i)}$ show that the
homomorphisms
$$
i_*:H_{2k-1}(\mathrm{SO}(2)_{\text{tor}};\Z)\cong
\Q/\Z\longrightarrow H_{2k-1}(B\mathrm{SL}^\delta (2,\R);\Z)
$$
are injective for all $k\geq 2$.

$\mathrm{(ii)}$\ If we assume that Thurston's lost theorem
(see Remark \ref{rem:f} below) holds:
$$
\chi^3=0\in H^6(B\mathrm{Diff}_+^{\omega,\delta} S^1;\Q),
$$
then Theorem \ref{th:hatchi}\ $\mathrm{(ii)}$ implies that
$$
i_*:H_{2k-1}(\mathrm{SO}(2)_{\text{tor}};\Z)\cong
\Q/\Z\longrightarrow H_{2k-1}(B\mathrm{Diff}_+^{\omega,\delta} S^1;\Z)
$$
are injective for all $k\geq 3$.
\end{example}

More generally, if some power $\chi^k$ of the Euler class vanishes rationally
for real analytic flat $S^1$-bundles, then by Remark \ref{rem:Q}, Theorem \ref{th:hatchi},
and Theorem \ref{th:other},
we can conclude that 
$H_{*}(B\mathrm{Diff}_+^{\omega,\delta} S^1;\Z)$ should have numerous torsion homology classes
which vanish in $H_{*}(B\mathrm{Diff}_+^{\delta} S^1;\Z)$.

On the other hand, we have the following.

\begin{cor}
Let $n\geq 2$ be a natural number and consider the subgroup 
$\Z/n\Z\subset \mathrm{Diff}^{\omega,\delta}_+S^1$ generated by the $1/n$-rotation of $S^1$.
Assume that the homomorphism
$$
H_{2k-1}(\Z/n\Z;\Z)\cong \Z/n\Z\rightarrow H_{2k-1}(B\mathrm{Diff}^{\omega,\delta}_+S^1;\Z)
$$
is trivial. Then 
$$
\chi^{k}\not=0\in H^{2k}(B\mathrm{Diff}^{\omega,\delta}_+S^1;\Q).
$$
\end{cor}

\begin{remark}
Mather \cite{MR283817, MR356085, MR356129} 
proved that $B\overline{\Gamma}_1$ is $2$-connected. On the other hand,
Haefliger \cite{MR100269} proved that $B\overline{\Gamma}^\omega_1$ is a $K(\pi,1)$ space
and $H_1(B\overline{\Gamma}^\omega_1;\Z)=0$. Thus the homotopy types of $B\overline{\Gamma}_1$
and $B\overline{\Gamma}^\omega_1$ are drastically different from each other.
Nevertheless, nothing is known at present
whether the natural map $B\overline{\Gamma}^\omega_1\rightarrow B\overline{\Gamma}_1$
induces isomorphism on homology or not.
Tsuboi \cite{MR3726712} proposed
to study $H_2(B\overline{\Gamma}^\omega_1;\Z)$. 

Also, it is an extremely difficult open problem to determine whether the
homomorphism
$$
H^*_{GF}(\mathcal{X}_{S^1},\mathrm{SO}(2))\cong \R[\alpha,\chi]/(\alpha\chi)\rightarrow H^*(B\mathrm{Diff}_+^{\omega,\delta}S^1;\R)
$$
from the Gel'fand-Fuchs cohomology of $S^1$ relative to $\mathrm{SO}(2)$
to the cohomology of $B\mathrm{Diff}_+^{\omega,\delta}S^1$
is injective or not. For
the case of $C^\infty$ diffeomorphism group, it was proved in \cite{MR759480} 
that it is injective. However, in the real analytic case, the only known results
concerning this problem is due to Thurston (\cite{MR298692}) who proved the continuous variability
of the $\alpha$-class and hence the linear independence of the classes
$\alpha,\chi$ (cf. Table \ref{tab:gysin} below).
The present work is a tiny attempt to attack this problem, in particular the
question of 
non-triviality of the rational $\chi^2$. Note that Ghys \cite{ghys1, ghys2} mentioned a tale of, what he called,
Thurston's lost theorem saying that $\chi^3=0\in H^6(B\mathrm{Diff}_+^{\omega,\delta}S^1;\Q)$.
See also Nariman \cite{nariman} where the author showed that any power of the {\it integral}
Euler class is non-trivial in $H^*(B\mathrm{Diff}_+^{\omega,\delta}S^1;\Z)$. 
These are our main motivation for this work.
\label{rem:f}
\end{remark}

\begin{table}[h]
\caption{$\text{Gysin exact sequence where
 $G=\mathrm{Diff}_+S^1, \tilde{G}=\widetilde{\mathrm{Diff}}_+S^1$}$}
\begin{center}
\begin{tabular}{|c|c|c|c|}
\noalign{\hrule height0.8pt}
\hfil $$ & $\text{smooth case}$ 
& {}  & \text{real analytic case} \\
\hline
$H_4(BG^\delta;\Z)$ & $\overset{(\alpha^\xi\alpha^\eta,\chi^2)}{\twoheadrightarrow}\ S^2_\Q(\R)\oplus\Z$ & $$ & 
$\overset{(\alpha^\xi\alpha^\eta,\chi^2)}{\rightarrow}\ S^2_\Q(\R)\oplus\Z
\ \text{what is the image?}$ \\
\hline
$\downarrow\cap\chi$ & $\downarrow$ & $$ & $\downarrow$ \\
\hline
$H_2(BG^\delta;\Z)$ & $\cong H_2(\Omega B\overline{\Gamma}_1;\Z)\oplus \Z
\overset{(\alpha,\chi)}{\twoheadrightarrow}\ \R\oplus\Z$ & $$ & 
$\overset{?}{\cong} H_2(\Omega B\Gamma_H^+;\Z)\oplus\Z\overset{(\alpha,\chi)}{\twoheadrightarrow}\ \R\oplus\Z$ \\
\hline
$\downarrow\mu$ & $\downarrow (\cong,\hspace{1cm} 0)\hspace{1cm}$ & $$ & $\downarrow$ \\                                                          
\hline 
$H_3(B\tilde{G}^\delta;\Z)$ & $\cong H_3(B\overline{\Gamma}_1;\Z)\oplus H_3(\Omega B\overline{\Gamma}_1;\Z)
\overset{\beta}{\twoheadrightarrow}\ \R$  & $$ & $\overset{?}{\cong}H_3(B\Gamma_H;\Z) \overset{\beta}{\twoheadrightarrow}\ \R$ \\
\hline 
$\downarrow p_*$ & $\downarrow$  & $$ & $\downarrow$ \\
\hline 
$H_3(BG^\delta;\Z)$ & $\cong H_3(\Omega B\overline{\Gamma}_1;\Z)\overset{?}{=}0$   & $$ & 
$\overset{?}{=}0$ \\
\hline 
$\downarrow\cap\chi$ & $\downarrow$   & $$ & $\downarrow$ \\
\hline 
$H_1(BG^\delta;\Z)$ & $0$  & $$ & $0$ \\
\hline 
$\downarrow \mu$ & $\downarrow$   & $$ & $\downarrow$ \\
\hline 
$H_2(B\tilde{G}^\delta;\Z)$ & $\cong H_2(\Omega B\overline{\Gamma}_1;\Z)
\overset{\alpha}{\twoheadrightarrow}\R$   & $$ & 
$\overset{?}{\cong} H_2(\Omega B\Gamma_H^+;\Z)
\overset{\alpha}{\twoheadrightarrow}\R$ \\
\hline 
$\downarrow p_*$ & $\downarrow$   & $$ & $\downarrow$ \\
\hline 
$H_2(BG^\delta;\Z)$ & $\cong H_2(\Omega B\overline{\Gamma}_1;\Z)\oplus \Z
\overset{(\alpha,\chi)}{\twoheadrightarrow}\ \R\oplus\Z$ & $$ & $\overset{(\alpha,\chi)}{\twoheadrightarrow}\ \R\oplus\Z$ \\
\hline 
$\downarrow\cap\chi$ & $\downarrow$   & $$ & $\downarrow$ \\
\hline 
$H_0(BG^\delta;\Z)$ & $\Z$  & $$ & $\Z$ \\
\hline 
$\downarrow\mu$ & $\downarrow$   & $$ & $\downarrow$ \\
\hline 
$H_1(BG^\delta;\Z)$ & $0$  & $$ & $0$ \\
\noalign{\hrule height0.8pt}
\end{tabular}
\end{center}
\label{tab:gysin}
\end{table}

\newpage
 
\vspace{2mm}
\section{Proofs}
In this section, we denote the group
$\mathrm{Diff}_+S^1$ (resp. $\widetilde{\mathrm{Diff}}_+S^1$)
by $G$ (resp. $\tilde{G}$) 
for simplicity.
First we prepare 
a few facts.

\begin{prop}
Let $X$ be a $2$-connected
topological space and let $\wedge X$ denote the
free loop space of $X$. Then we have
\begin{align*}
H_2(\wedge X;\Z)&\cong
H_2(\Omega X;\Z)\cong H_3(X;\Z)\\
H_3(\wedge X;\Z)&\cong H_3(X;\Z)
\oplus H_3(\Omega X;\Z).
\end{align*}
\label{prop:gh23}
\end{prop}

\begin{proof}
This can be shown by the usual spectral sequence arguments
applied to the fibration
$
\Omega X\rightarrow \wedge X\rightarrow X
$
which has a section. The isomorphism $H_2(\Omega X;\Z)\cong H_3(X;\Z)$
is induced by the theorem of Hurewicz as
$$
H_2(\Omega X;\Z)\cong \pi_2(\Omega X)\cong \pi_3(X)\cong H_3(X;\Z).
$$
\end{proof}

\begin{prop}[well known, see for example \cite{MR2039760}]
$$
H_*(\wedge S^3;\Z)\cong H_*(\Omega S^3;\Z)\otimes H_*(S^3;\Z)
\cong \Z[\alpha]\otimes\wedge (\beta)
$$
where $\alpha\in H_2(\Omega S^3;\Z)\cong\Z$ and $\beta\in H_3(S^3;\Z)\cong\Z$
are generators.
\label{prop:wk}
\end{prop}

\begin{thm}[Haefliger \cite{privcom}, 
Nariman \cite{nariman}, see Remark 1.14.]
For any $k$, there exists certain element $\sigma_k\in H_{2k}(BG^\delta;\Z)$
such that $\chi^k(\sigma_k)=1$.
\label{prop:chi1}
\end{thm}

Haefliger first pointed out that the mapping $H$ of Theorem \ref{th:Thurston}
is $S^1$-equivariant with respect to natural actions of $S^1$ on
both sides. Then the required claim follows from Theorem \ref{th:Thurston} by using the fact that the
$S^1$ action on free loop spaces has fixed points.

\vspace{2mm}
Consider the following part of the Gysin exact sequence associated with the central extension \eqref{eq:ce}
$$
\cdots\rightarrow H_{k+2}(BG^\delta;\Z)\overset{\cap\chi}{\longrightarrow}
H_{k}(BG^\delta;\Z) \overset{\mu}{\rightarrow} H_{k+1}(B\tilde{G}^\delta;\Z)
\overset{p_*}{\longrightarrow} H_{k+1}(BG^\delta;\Z)\rightarrow \cdots.
$$

\begin{prop}
Define a homomorphism
$$
\nu: H_k(\wedge B\overline{\Gamma}_1;\Z)
\rightarrow H_{k+1}(\wedge B\overline{\Gamma}_1;\Z)
$$
by the following map
$$
H_k(\wedge B\overline{\Gamma}_1;\Z)\ni \tau\mapsto
\nu(\tau)=\theta'_*(\tau\times [S^1])\in H_{k+1}(\wedge B\overline{\Gamma}_1;\Z)
$$
where $\theta': \wedge B\overline{\Gamma}_1\times S^1\rightarrow\wedge B\overline{\Gamma}_1$
denotes the natural $S^1$ action on $\wedge B\overline{\Gamma}_1$.
Then the following diagram is commutative:
$$
\begin{CD}
H_{k}(B\tilde{G}^\delta;\Z)
 @>{H_* }>{\cong}> H_k(\wedge B\overline{\Gamma}_1;\Z)
\\
@V{\mu\circ p_*}VV  @VV{\nu}V\\
H_{k+1}(B\tilde{G}^\delta;\Z) @>{H_*}>{\cong}>H_{k+1}(\wedge B\overline{\Gamma}_1;\Z).
\end{CD}
$$
\label{prop:nu}
\end{prop}

\begin{proof}
As already mentioned above, the mapping
$
H: B\tilde{G}^\delta\rightarrow \wedge B\overline{\Gamma}_1
$
is $S^1$-equivariant with respect to the natural $S^1$-actions.
The $S^1$-action 
$$
\theta: B\tilde{G}^\delta\times S^1\rightarrow B\tilde{G}^\delta
$$
on $B\tilde{G}^\delta$ is defined because $B\tilde{G}^\delta$ can be considered as
the total space of the universal $S^1$-bundle over $BG^\delta$, which is an $S^1$-
principal bundle. Thus we have the following homotopy commutative diagram
$$
\begin{CD}
B\tilde{G}^\delta\times S^1
@>{ H\times\text{id}}>> \wedge B\overline{\Gamma}_1\times S^1\\
@V{\theta}VV @VV{\theta'}V\\
B\tilde{G}^\delta @>{H}>> \wedge B\overline{\Gamma}_1
\end{CD}
$$

Now the homomorphism
$$
\mu\circ p_*: H_{k}(B\tilde{G}^\delta;\Z)\rightarrow H_{k+1}(B\tilde{G}^\delta;\Z)
$$
is realized by the map:
$$
H_{k}(B\tilde{G}^\delta;\Z)\ni \sigma\longmapsto 
\mu\circ p_*(\sigma)=
\theta_*(\sigma\times [S^1]) \in
H_{k+1}(B\tilde{G}^\delta;\Z).
$$
The claim follows from this.
\end{proof}

\vspace{2mm}
\noindent
{\it Proof of Theorem \ref{th:main}}\
As already mentioned, Mather 
proved that $B\overline{\Gamma}_1$ is $2$-connected.
If we apply Proposition \ref{prop:gh23} to the case $X=B\overline{\Gamma}_1$, then we obtain the second statements of
$\mathrm{(i)}$ and $\mathrm{(ii)}$.

Next we prove the first statement of $\mathrm{(i)}$.
Consider the following two parts of the Gysin exact sequence of the central extension \eqref{eq:ce}:
\begin{align}
&0=H_1(BG^\delta;\Z) \overset{\mu}{\rightarrow} H_2(B\tilde{G}^\delta;\Z)
\overset{p_\ast}{\rightarrow} H_2(BG^\delta;\Z)\overset{\cap\chi}{\longrightarrow} H_0(BG^\delta;\Z)=\Z
\rightarrow 0
\label{eq:1}\\
&
\rightarrow H_4(BG^\delta;\Z) \overset{\cap \chi}{\longrightarrow}
H_2(BG^\delta;\Z) \overset{\mu}{\rightarrow} H_3(B\tilde{G}^\delta;\Z)
\overset{p_\ast}{\rightarrow }H_3(BG^\delta;\Z)\rightarrow H_1(BG^\delta;\Z)=0
\label{eq:2}
\end{align}
where $\chi\in H^2(BG^\delta;\Z)$ denotes the Euler class and
$H_1(BG^\delta;\Z)=0$ is due to Hermann \cite{herman1}.
From the exact sequence \eqref{eq:1}, we have
$$
H_2(BG^\delta;\Z)\cong H_2(B\tilde{G}^\delta;\Z)\oplus \Z\quad (\text{non-canonical}).
$$
On the other hand, we show that the restriction $\mu\circ p_\ast$ of the homomorphism $\mu$ in the
exact sequence \eqref{eq:2} to the submodule $H_2(B\tilde{G}^\delta;\Z)\subset H_2(BG^\delta;\Z)$ is injective.
In view of Proposition \ref{prop:nu} together with the second statements of,
we can write
$$
\begin{CD}
H_{2}(B\tilde{G}^\delta;\Z)
 @>{H_*}>{\cong}> H_2(\wedge B\overline{\Gamma}_1;\Z)\cong H_2(\Omega B\overline{\Gamma}_1;\Z)
\\
@V{\mu\circ p_*}VV  @VV{\nu}V\\
H_{3}(B\tilde{G}^\delta;\Z) @>{H_*}>{\cong}>H_{3}(\wedge B\overline{\Gamma}_1;\Z)
\cong H_3(B\overline{\Gamma}_1;\Z)\oplus
H_3(\Omega B\overline{\Gamma}_1;\Z).
\end{CD}
$$
For any element $\sigma\in H_2(B\tilde{G}^\delta;\Z)$, we consider the element
$$
\tau=H_*(\sigma)\in H_2(\wedge B\overline{\Gamma}_1;\Z)\cong H_2(\Omega B\overline{\Gamma}_1;\Z).
$$

Let $\tau^*\in H_3(B\overline{\Gamma}_1;\Z)\cong \pi_3 (B\overline{\Gamma}_1)$
be the element which corresponds to $\tau\in H_2(\Omega B\overline{\Gamma}_1;\Z)$
under the natural isomorphism $H_3(B\overline{\Gamma}_1;\Z)\cong H_2(\Omega B\overline{\Gamma}_1;\Z)$ and choose a continuous mapping 
$$
f: S^3\rightarrow B\overline{\Gamma}_1
$$
which represents $\tau^*$. Then,
if we denote by $\iota: S^2\rightarrow \Omega S^3$
the mapping representing the generator of $\pi_2 (\Omega S^3)\cong \Z$,
the above element $\tau=H_*(\sigma)$
is represented by the following composed mapping
$$
S^2\overset{\iota}{\longrightarrow}\Omega S^3
\overset{\Omega f}{\longrightarrow} \Omega B\overline{\Gamma}_1
\subset \wedge B\overline{\Gamma}_1.
$$
By Proposition \ref{prop:nu} again,
$$
H_*(\bar{\mu}(\sigma))=\nu(H_*(\sigma))=\nu(\tau)=\theta'_*(\tau\times [S^1]).
$$
On the other hand, the following diagram is clearly commutative
$$
\begin{CD}
\wedge S^3\times S^1
@>{\wedge f\times\text{id} }>> \wedge B\overline{\Gamma}_1\times S^1\\
@V{\theta^{\prime\prime}}VV @VV{\theta'}V\\
\wedge S^3 @>{\wedge f}>> \wedge B\overline{\Gamma}_1
\end{CD}
$$
where $\theta^{\prime\prime}$ denotes the $S^1$ action on $\wedge S^3$.
Hence
$$
\theta'_*(\tau\times [S^1])=(\wedge f)_*(\theta^{\prime\prime}_*( j_*\iota_*([S^2])\times [S^1]))
$$
where $j: \Omega S^3\rightarrow \wedge S^3$ denotes the inclusion.
In the terminology of Proposition \ref{prop:wk}, we have
$$
j_*\iota_*([S^2])=\alpha.
$$
On the other hand, it is easy to see that
$$
\theta^{\prime\prime}_*(\alpha\times [S^1])=\beta.
$$
Hence
$$
\nu(\tau)=\theta'_*(\tau\times [S^1])=(\wedge f)_*(\beta)
$$
and finally
$$
(\wedge f)_*(\beta)=f_*([S^3])=\tau^*\in H_3(B\overline{\Gamma}_1;\Z)\subset 
H_3(\wedge B\overline{\Gamma}_1;\Z).
$$
Summing up, the homomorphism 
$\bar{\mu}:H_{2}(B\tilde{G}^\delta;\Z)\rightarrow H_{3}(B\tilde{G}^\delta;\Z)$
is described, under the isomorphism $H_*$, as
$$
H_2(\wedge B\overline{\Gamma}_1;\Z)
\cong H_2(\Omega B\overline{\Gamma}_1;\Z)
\ni
\tau\overset{\nu}{\mapsto} (\tau^*,0)
\in H_3(B\overline{\Gamma}_1;\Z)\oplus
H_3(\Omega B\overline{\Gamma}_1;\Z).
$$
Observe here that $\nu$ does not hit the second component because
$H_3(\Omega S^3;\Z)=0$. 

Thus we have proved that $\bar{\mu}$ is injective as required.
Now we use Theorem \ref{prop:chi1} to conclude that the composition
$$
H_4(BG^\delta;\Z)\overset{\cap\chi}{\longrightarrow}
H_2(BG^\delta;\Z))\overset{\cap\chi}{\longrightarrow} H_0(BG^\delta;\Z)\cong\Z
$$
is surjective. Now consider the following subgroup
$$
\mathrm{Im}\left(H_4(BG^\delta;\Z)\overset{\cap\chi}{\longrightarrow}
H_2(BG^\delta;\Z)\right)=
\mathrm{Ker}\left(H_2(BG^\delta;\Z)\overset{\mu}{\longrightarrow}
H_3(B\tilde{G}^\delta;\Z)\right)
$$
of $H_2(BG^\delta;\Z)$. Then it is easy to see that this subgroup is isomorphic
to
$\Z$ and we have the required canonical direct sum decomposition
\begin{equation}
H_2(BG^\delta;\Z)=H_2(B\tilde{G}^\delta;\Z)\oplus\Z
\cong H_3(B\overline{\Gamma}_1;\Z)\oplus\Z\quad (\text{canonical direct sum}).
\label{eq:cds}
\end{equation}
This finishes the proof of $\mathrm{(i)}$ and also $\mathrm{(iii)}$.
Finally the first statement of $\mathrm{(ii)}$ follows from 
the exact sequence \eqref{eq:2} and the second statement of $\mathrm{(ii)}$.
Thus we have proved Theorem \ref{th:main}.
\qed

\vspace{2mm}
\noindent
{\it Proof of Theorem \ref{th:phi}}\quad
By Theorem \ref{th:main}, the second factor of $\mu(\sigma)$ in
$$
H_3(B\tilde{G}^\delta;\Z)\cong H_3(B\overline{\Gamma}_1;\Z)\oplus
H_3(\Omega B\overline{\Gamma}_1;\Z)
$$
is trivial and the first factor is detected by the 
following composed mapping
$$
B\tilde{G}^\delta \rightarrow \wedge B\overline{\Gamma}_1
\rightarrow B\overline{\Gamma}_1
$$
where the second one denotes the evaluation map at the base point of
$S^1$ which also serves as the natural retraction  onto the
subspace of $\wedge B\overline{\Gamma}_1$ consisting of constant loops.
This composed mapping factors through the classifying space of
$\mathrm{Diff}^\delta_+\R$ as
$$
B\tilde{G}^\delta\rightarrow B\mathrm{Diff}^\delta_+\R\rightarrow 
B\overline{\Gamma}_1.
$$
The endomorphism $\varphi_k$ of $\tilde{G}^\delta$ is defined on this
larger group $\mathrm{Diff}^\delta_+\R$ as an inner automorphism.
Since any inner automorphism of a group acts on the homology trivially,
the required result follows.
\qed

\vspace{2mm}
Here we recall a few facts from \cite{MR759480}.
Let
$
p_k: G^{(k)}\rightarrow G
$
be the $k$-fold cover of $G$. Then it can be described as
$$
G^{(k)}=\{f\in G; fR(1/k)=R(1/k)f \} 
$$
where $R(1/k)$ denotes the rotation by $1/k$, so that
we have also the inclusion
$$
i_k: G^{(k)}\subset G.
$$
\begin{definition}
Define an endomorphism
$$
\varphi^\Q_k: H_*(BG^\delta;\Q)\rightarrow H_*(BG^\delta;\Q)
$$
by setting
$$
\varphi^\Q_k(\sigma)= (i_k)_*(p_k)_*^{-1} (\sigma)\quad
(\sigma\in H_m(BG^\delta;\Q)).
$$
\label{def:phi}
\end{definition}

\begin{remark}
The above definition is given by adapting the endomorphism $\varphi_k$ on $\tilde{G}$
to the case of $G$ but only at the rational  homological level.
In fact, it can be seen that the following diagram is commutative.
$$
\begin{CD}
 H_m(B\tilde{G}^\delta;\Q) @>{(\varphi_k)_*}>> H_m(B\tilde{G}^\delta;\Q)\\
 @V{p_*}VV  @VV{p_*}V\\
  H_{m}(BG^\delta;\Q) @>{\varphi^\Q_k}>> H_{m}(BG^\delta;\Q).
\end{CD}
$$
\label{re:phi}
\end{remark}

\begin{prop}[\cite{MR759480}]
For any $\sigma\in H_m(BG^\delta;\Q)$, we have the identity
$$
(\varphi_k)_*(\mu(\sigma))=\mu\Bigl(\frac{1}{k}\varphi^\Q_k(\sigma)\Bigr).
$$
Namely, the following diagram is commutative
$$
\begin{CD}
 H_m(BG^\delta;\Q) @>{\frac{1}{k}\varphi^\Q_k}>> H_m(BG^\delta;\Q)\\
 @V{\mu}VV  @VV{\mu}V\\
  H_{m+1}(B\tilde{G}^\delta;\Q) @>{\varphi_k}>> H_{m+1}(B\tilde{G}^\delta;\Q).
\end{CD}
$$
This also holds if we replace $G^\delta,\tilde{G}^\delta$  by 
$G^{\omega,\delta},\tilde{G}^{\omega,\delta}$, respectively.
\label{prop:phi}
\end{prop}

\begin{prop}[Action of $\varphi_k^\Q$]
In the direct sum decomposition
$$
H_2(BG^\delta;\Q)=H_2(B\tilde{G}^\delta;\Q)\oplus \Q\quad (\text{canonical direct sum})
$$
given in the proof of Theorem \ref{th:main} (see \eqref{eq:cds}), both summands
$H_2(B\tilde{G}^\delta;\Q)$ and $\Q$ are eigenspaces of $\varphi_k^\Q$ with eigenvalues 
$k$ and $\frac{1}{k}$, respectively.
\label{prop:phi_k^Q}
\end{prop}

\vspace{3mm}
\noindent
{\it Proof of Proposition \ref{prop:chi}}

By the assumption,
$$
(\varphi_k)_* (\mu(\sigma))=\mu(\sigma)\in 
H_3(B\widetilde{\mathrm{Diff}}_+^{\omega,\delta} S^1;\Z).
$$
for some
$\sigma\in H_2(B\mathrm{Diff}_+^{\omega,\delta} S^1;\Z)$ with $\chi(\sigma)\not=0$ and 
$k$.
Then by Proposition \ref{prop:phi}, we have
$$
\mu\Bigl(\sigma-\frac{1}{k}\varphi_k^\Q(\sigma)\Bigr)=0.
$$
Hence, by the exact sequence $\eqref{eq:2}$ 
(for the real analytic diffeomorphism group with the rational coefficients), we have
$$
\Bigl(\sigma-\frac{1}{k}\varphi_k^\Q(\sigma)\Bigr)\in
\mathrm{Im}\,\Bigl(H_4(B\mathrm{Diff}_+^{\omega,\delta} S^1;\Q) 
\overset{\cap \chi}{\longrightarrow}
H_2(B\mathrm{Diff}_+^{\omega,\delta} S^1;\Q)\Bigr).
$$
On the other hand, we have
$$
\chi\Bigl(\sigma-\frac{1}{k}\varphi_k^\Q(\sigma)\Bigr)=
\Bigl(1-\frac{1}{k^2}\Bigr)\chi(\sigma)\not=0.
$$
Therefore $\chi^2\not=0\in H^4(B\mathrm{Diff}_+^{\omega,\delta} S^1;\Q)$ completing the proof.
\qed

\vspace{3mm}
\noindent
{\it Proof of Theorem \ref{th:Q}}

The short exact sequence \eqref{eq:1} holds also for the real analytic case:
$$
0=H_1(BG^{\omega,\delta};\Z) \overset{\mu}{\rightarrow} H_2(B\tilde{G}^{\omega,\delta};\Z)
\rightarrow H_2(BG^{\omega,\delta};\Z)\overset{\cap\chi}{\longrightarrow} H_0(BG^{\omega,\delta};\Z)=\Z
\rightarrow 0
$$
where $G^{\omega,\delta}$ denotes $\mathrm{Diff}_+^{\omega,\delta} S^1$ and $H_1(BG^{\omega,\delta};\Z)=0$
because of Herman's result \cite{herman1} that $\mathrm{Diff}_+^{\omega,\delta} S^1$ is a simple group.
Therefore 
$$
H_2(BG^{\omega,\delta};\Z)\cong
H_2(B\tilde{G}^{\omega,\delta};\Z)\oplus \Z\quad (\text{non-canonical direct sum}).
$$
By the assumption $\chi^2=0$ and the exact sequence \eqref{eq:2} 
$$
\rightarrow H_4(BG^{\omega,\delta};\Z) \overset{\cap \chi}{\longrightarrow}
H_2(BG^{\omega,\delta};\Z) \overset{\mu}{\rightarrow} H_3(B\tilde{G}^{\omega,\delta};\Z)
\rightarrow H_3(BG^{\omega,\delta};\Z)\rightarrow H_1(BG^{\omega,\delta};\Z)=0
$$
for the real analytic case, we can conclude that the homomorphism $\mu$ is {\it injective} on 
the (non-canonical) summand $\Z$. Furthermore
$$
\mu(H_2(B\tilde{G}^{\omega,\delta};\Z))\cap \mu(\Z)=\{0\}
$$
because otherwise, there is a non-zero integer $n$ and
$y\in H_2(B\tilde{G}^{\omega,\delta};\Z))$ such that 
$\mu(y)=\mu(n)$. It follows that $\mu(y-n)=0$ which implies $y-n\in\Im\,\cap \chi$.
But since $\chi(y-n)=-n\not=0$, this contradicts the assumption $\chi^2=0$.
Hence
$$
\mathrm{Im}\,\mu =\mu(H_2(B\tilde{G}^{\omega,\delta};\Z))\oplus \mu(\Z)\ \subset 
H_3(B\tilde{G}^{\omega,\delta};\Z)
$$
where $\mu(\Z)\cong\Z$.
Let us define
\begin{align*}
\tilde{P}=&\text{submodule of $H_3(B\tilde{G}^{\omega,\delta};\Z)$ generated by
$\mu(H_2(B\tilde{G}^{\omega,\delta};\Z))$ and}\\
&
\text{the elements $\{\mu(1),(\varphi_k)_*(\mu(1));k=2,3,\cdots\}$}
\end{align*}
and set
$$
P=\tilde{P}/\mu(H_2(B\tilde{G}^{\omega,\delta};\Z)).
$$
We construct a surjective homomorphism
\begin{equation}
T_P: P\twoheadrightarrow\Q.
\label{eq:hat}
\end{equation}
By the definition of the group $P$, there is a homomorphism
$$
H_2(BG^{\omega,\delta};\Z)/H_2(B\tilde{G}^{\omega,\delta};\Z)\cong \Z\rightarrow P
$$
and the left hand side is detected by the homomorphism 
$\chi:H_2(BG^{\omega,\delta};\Z)\rightarrow Z$.
Consider the rational form of this homomorphism
\begin{equation}
H_2(BG^{\omega,\delta};\Q)/H_2(B\tilde{G}^{\omega,\delta};\Q)\cong \Q\rightarrow P\otimes_\Z \Q.
\label{eq:hatt}
\end{equation}
By Proposition \ref{prop:phi}, together with Definition \ref{def:phi} and Remark \ref{re:phi},
the homomorphism $\varphi_k^\Q$ acts on the left hand side with eigenvalue $\frac{1}{k}$, and
this homomorphism is transferred to the homomorphism $k(\varphi_k)_*$ on the right hand side.
Here observe that this operation preserves the space $\mu(H_2(B\tilde{G}^{\omega,\delta};\Q))$
so that it induces that on the quotient $P\otimes_\Z\Q$.
Now we show that the above homomorphims \eqref{eq:hatt} is an isomorphism.
Injectivity is clear because the generator of the summand $\Z\ (\text{non-canonical})\subset 
H_2(G^{\omega,\delta};\Q)$ goes to $\mu(1)\otimes 1$. To prove the surjectivity, it is 
enough to show that the element $(\varphi_k)_*(\mu(1))\otimes 1$ is contained in the image
for any $k$.
But by Proposition \ref{prop:phi} and \ref{prop:phi_k^Q}, 
$$
(\varphi_k)_*(\mu(1))\otimes 1 = \mu \left(\frac{1}{k}\varphi_k^\Q (1)\right)\otimes 1
=\mu(1)\otimes \frac{1}{k^2}.
$$
Now we define the homomorphism 
$T_P$ to be the composition
$$
T_P: P\rightarrow P\otimes_\Z\Q\ \overset{\eqref{eq:hatt}}{\cong}\ \Q.
$$
It remains to prove that this homomorphism is surjective
(this is the main point of the proof). For any rational number
$\frac{p}{q}\in\Q$, consider the element $pq\ (\varphi_q)_*(\mu(1))\in P$. Then we have
$$
T_P\left(pq\ (\varphi_q)_*(\mu(1))\right) 
=
T_P\left( pq\ \mu \left(\frac{1}{q}\varphi_q^\Q (1)\right) \right)
=
T_P\left( pq\ \frac{1}{q^2}\mu (1) \right)=\frac{p}{q}
$$
proving the surjectivity.
Finally the last claim $P/\Z\subset H_3(BG^{\omega,\delta};\Z)$ follows from the 
exact sequence \eqref{eq:2}.
\qed

\vspace{2mm}
\begin{remark}
The homomorphism $T_P$ constructed above
can be interpreted as a particular case of the secondary characteristic class defined
in Definition \ref{def:sc} below. In this case, it is equal to the class
$T{\chi^2}$ defined
in $H^3(B\tilde{G}^{\omega,\delta};\Q)$ associated with the assumption that 
$\chi^2=0\in H^4(BG^{\omega,\delta};\Q)$. 

Apart from the above, one distinctive feature here is that
the value of this homomorphism takes any rational number on
{\it integral} homology classes in $H_3(B\tilde{G}^{\omega,\delta};\Z)$.
\end{remark}

\vspace{2mm}
\noindent
{\it Proof of Theorem \ref{th:other}}.  

Since $\displaystyle 
\mathrm{SO}(2)_{\text{tor}}=\lim_{\underset{n}{\rightarrow}}\Z/n\Z$
it is enough to show that the homomorphism 
$$
H_{2k-1}(\Z/n\Z;\Z)\cong \Z/n\Z\rightarrow H_{2k-1}(BG^\delta;\Z)
$$
is trivial for any $n,\, k\in\mathbb{N}$ . 
This homomorphism is induced by a mapping $i: L^{2k-1}_n\rightarrow BG^\delta$ 
from the $(2k-1)$-dimensional lens space $L^{2k-1}_n=S^{2k-1}/(\Z/n\Z)$ to $BG^\delta$ defined as the 
composition $L^{2k-1}_n \rightarrow B(\Z/n\Z) \rightarrow BG^\delta$.
Let 
$$
S^1\rightarrow L_n^{2k-1}\tilde{\times} S^1\rightarrow L_n^{2k-1}
$$
be the foliated $S^1$-bundle
over $L_n^{2k-1}$ corresponding to the mapping $i$, where
$L_n^{2k-1}\tilde{\times} S^1$ denotes its total space.
Since, as already mentioned, $B\tilde{G}^\delta$ can be considered as the total
space of the universal flat $S^1$-bundle over $BG^\delta$, there exists a map
$$
\tilde{i}: L_n^{2k-1}\tilde{\times} S^1\rightarrow B\tilde{G}^\delta
$$
making the following diagram
commutative
$$
\begin{CD}
 L_n^{2k-1}\tilde{\times} S^1 @>{\tilde{i}}>> B\tilde{G}^\delta\\
 @VVV  @VVV \\ 
 L_n^{2k-1} @>{i}>> BG^\delta.
\end{CD}
$$
The foliation on 
$$
L_n^{2k-1}\tilde{\times} S^1= S^{2k-1}\times_{\Z/n} S^1
$$
can also be described as 
the quotient of the horizontal foliation 
on $S^{2k-1}\times S^1$ by the action 
of $\Z/n\Z$, where the generator of $\Z/n\Z$ acts on $S^1$ by $1/n$ rotation.  
Hence, its leaf space is considered to be 
a circle which is denoted by $S^1/n=S^{1}/(\Z/n\Z)$ and 
the mapping 
$$
f: L_n^{2k-1}\tilde{\times} S^1\rightarrow S^1/n
$$
to the leaf space 
restricts 
to each fiber as an $n$-fold covering map. 

Thanks to the following Proposition, essentially due to Haefliger
\cite{privcom} 
and Nariman \cite{nariman}, 
it suffices to prove 
\[
(H/\!/S^{1})_{*}\circ i_{*}([L_{n}^{2k-1}])=0
\in H_{2k-1}(\wedge B\overline{\Gamma}_1/\!/S^{1};\Z)
\]
instead of showing $ i_{*}([L_{n}^{2k-1}])=0 
\in H_{2k-1}(BG^{\delta};\Z)$, 
where $H/\!/S^{1}:B\tilde{G}^{\delta}/\!/S^{1} \rightarrow
\wedge B\overline{\Gamma}_1/\!/S^{1}$ is the Borel 
$S^{1}$-quotient map associated with 
the Mather-Thurston map 
$H:B\tilde{G}^{\delta} \rightarrow
\wedge B\overline{\Gamma}_1$ and 
$B\tilde{G}^{\delta}/\!/S^{1}$ is replaced with 
$B{G}^{\delta}$ because they are homotopically equivalent. 

\begin{prop}\label{prop:Borel}
{\rm (i) 
}
The Mather-Thurston map 
$H:B\tilde{G}^{\delta} \rightarrow
\wedge B\overline{\Gamma}_1$ 
in Thurston's Theorem \ref{th:Thurston} is $S^{1}$-equivariant. 
\\
{\rm (ii)} The Borel $S^{1}$-quotient map 
$$H/\!/S^{1}:B\tilde{G}^{\delta}/\!/S^{1} \rightarrow
\wedge B\overline{\Gamma}_1/\!/S^{1}$$ 
associated with 
the Mather-Thurston map $H$ induces isomorphisms between their homology groups. 
\end{prop}

\noindent
{\it Proof of Proposition \ref{prop:Borel}} (ii). We enlarge 
the Mather-Thurston map $H$ to 
$H:B\tilde{G}^{\delta}\times ES^1 \rightarrow
\wedge B\overline{\Gamma}_1\times ES^1$.  
Thanks to (i) the enlarged $H$ is still $S^{1}$-equivariant and 
induces isomorphisms between homologies. 
$$
\begin{CD}
S^{1}@>>>B\tilde{G}^\delta\times ES^1@>>>B\tilde{G}^{\delta}/\!/S^{1}
\\
@|@V{H}V{\text{homology iso.}}V@VV{H/\!/S^{1}}V
\\
S^{1}@>>>\wedge B\overline{\Gamma}_1^\delta\times ES^1
@>>>\wedge B\overline{\Gamma}_1/\!/S^{1}
\end{CD}
$$
Comparing the homology Gysin sequences, we see 
$H/\!/S^{1}$ induces isomorphisms between homologies 
because so does $H$. 
This applies to the Borel quotient map associated with 
any $S^{1}$-equivariant mapping which induces isomorphisms between homology groups. 
\qed

In order to investigate the cycle $H/\!/S^{1}\circ i(L_{n}^{2k-1})$, 
we look at $H\circ \tilde{i}(L_n^{2k-1}\tilde{\times} S^1)$ and 
concider where it locates. 
As $H$ is the adjoint of the classifying map 
$$
h:B\tilde{G}^\delta\times S^1\rightarrow B\overline{\Gamma}_1
$$
of foliations of codimension one, we look at 
the foliation on $B\tilde{G}^\delta\times S^1$
induced 
by the projection 
$
B\tilde{G}^\delta\times S^1\rightarrow B\tilde{G}^\delta
$
from the one on $B\tilde{G}^\delta$, 
which gives rise to 
the universal flat $S^1$-product structure over $B\tilde{G}^\delta$, 
and the  foliation on  $(L_n^{2k-1}\tilde{\times} S^1)\times S^1$ 
induced from the one on  $L_n^{2k-1}\tilde{\times} S^1$ by the projection 
 $(L_n^{2k-1}\tilde{\times} S^1)\times S^1
 \rightarrow 
 L_n^{2k-1}\tilde{\times} S^1$.   
 The foliation on  $(L_n^{2k-1}\tilde{\times} S^1)\times S^{1}$ is 
 $S^{1}$-invariant 
 and the leaf space is again identified with $S^1/n$. 
 Therefore it is the pull-back 
 of the point foliation on  $S^1/n$ by 
 $$
 \tilde{f}: (L_n^{2k-1}\tilde{\times} S^1)\times S^{1}\rightarrow S^{1}/n
 $$ 
 and the right action 
 of $S^{1}$ induces the $n$-fold covering map on   $S^1/n$.   
 Hence the classifying map $h\circ \tilde{\tilde{i}}$ of this foliation 
is homotopic to  $\iota\circ\tilde{f}$, 
where   
$\tilde{\tilde{i}}:(L_n^{2k-1}\tilde{\times} S^1)\times S^1\rightarrow 
B\tilde{G}^\delta\times S^1$ 
covers the classifying map $\tilde{i}$ 
and 
$\iota: S^1\rightarrow B\overline{\Gamma}_1$ 
denotes the classifying map of the point foliation on $S^1/n$.
Since $B\overline{\Gamma}_1$ is simply connected (in fact $2$-connected as mentioned already), there exists a mapping
$\tilde{\iota}: D^2\rightarrow B\overline{\Gamma}_1$ which extends the mapping $\iota$ on $\partial D^2=S^1/n$.
Thus we obtain
the following homotopy commutative diagram
$$
\begin{CD}
 (L_n^{2k-1}\tilde{\times} S^1)\times S^1@>{\tilde{f}}>> 
 S^1/n\subset D^{2}
@>{\tilde{\iota}}>>  B\overline{\Gamma}_1
\\
@|@.@|
\\
 (L_n^{2k-1}\tilde{\times} S^1)\times S^1@>{\tilde{\tilde{i}}}>> 
B\tilde{G}^\delta\times S^1 @>{h}>> B\overline{\Gamma}_1
 \\
@VV{\text{right action}}V @VV{\text{right action}}V
\\ 
L_n^{2k-1}\tilde{\times} S^1 @>{\tilde{i}}>>  
B\tilde{G}^\delta @>{H}>> 
\wedge  B\overline{\Gamma}_1(\times ES^{1})
\\
@VV{\text{right action}}V @VV{\text{right action}}V @VV{\text{right action}}V
 \\
 L_n^{2k-1} @>{i}>>  BG^\delta \sim B\tilde{G}^\delta/\!/S^1
@>{H/\!/S^{1}}>>\wedge  B\overline{\Gamma}_1/\!/S^1 .  
\end{CD}
$$

Passing from $\tilde{\iota}\circ\tilde{f}$ 
to the adjoint map $H\circ\tilde{i}$, 
we obtain the following homotopy commutative diagram
$$
\begin{CD}
 L_n^{2k-1}\tilde{\times} S^1 @>>> \wedge D^2 @>>> \wedge B\overline{\Gamma}_1\\
 @V{i}VV   @.  @|\\ 
  B\tilde{G}^\delta @= B\tilde{G}^\delta @>{H}>> \wedge B\overline{\Gamma}_1.
\end{CD}
$$
Every space appearing in this diagram admits a natural $S^1$-action and every map here is 
$S^1$-equivariant. Therefore by taking the Borel construction on each space, we obtain the following homotopy
commutative diagram
$$
\begin{CD}
 (L_n^{2k-1}\tilde{\times} S^1)/S^1=L_n^{2k-1} @>>> \wedge D^2/\!/S^1\overset{\text{h.e.}}{\simeq} BS^1=\mathbb{C}P^\infty @>>> \wedge B\overline{\Gamma}_1/\!/S^1\\
 @V{\tilde{i}}VV   @.  @|\\ 
  B\tilde{G}^\delta/S^1=BG^\delta @= BG^\delta @>{H/\!/S^1}>{\text{homology iso.}}> \wedge B\overline{\Gamma}_1 /\!/S^1.
\end{CD}
$$

Now since $H_{2k-1}(\mathbb{C}P^\infty;\Z)=0$, Proposition \ref{prop:Borel} 
implies that
$i_*([L_n^{2k-1}])=0\in H_{2k-1}(BG^\delta;\Z)$ as required. This completes the proof.
\qed

\begin{remark}
We mention that some part of the above proof is valid in the real analytic case. 
Namely, the classifying map 
$f^\omega:L_n^{2k-1}\tilde{\times} S^1\rightarrow B\overline{\Gamma}_1^\omega$ factors as
$$
L_n^{2k-1}\tilde{\times} S^1 
\xrightarrow[\text{$n$-fold cover}]{\text{fiberwise}}
S^1
\overset{\iota^\omega}{\rightarrow} B\overline{\Gamma}_1^\omega
$$
where $\iota^\omega$ denotes the classifying map for the point foliation on $S^1$
which is real analytic. 
It follows immediately that
$$
f^\omega_*([L_n^{2k-1}\tilde{\times} S^1])=0\in H_{2k}(B\overline{\Gamma}^\omega_1;\Z).
$$
We also have the following.
$$
\tilde{i}^\omega_*([L_n^{2k-1}\tilde{\times} S^1])=0\in H_{2k}(B\tilde{G}^{\omega,\delta};\Z).
$$
To prove this, consider 
the following part of the Gysin exact sequence
$$
H_{2k+1}(BG^{\omega,\delta};\Z)\overset{\cap \chi}{\rightarrow} H_{2k-1}(BG^{\omega,\delta};\Z)
\overset{\mu}{\rightarrow}
H_{2k}(B\tilde{G}^{\omega,\delta};\Z).
$$
Then we have the equality
$$
i^\omega_*([L^{2k+1}_n])\cap \chi=i^\omega_*([L^{2k-1}_n]).
$$
This follows from the argument in the proof of Theorem \ref{th:hatchi} (i) given below.
By the exactness, we can conclude that
$$
\mu(i^\omega_*([L^{2k-1}_n]))=\tilde{i}^\omega_*([L_n^{2k-1}\tilde{\times} S^1])=0\in H_{2k}(B\tilde{G}^{\omega,\delta};\Z).
$$

The important problem is to determine whether the stronger condition
$$
i^\omega_*([L_n^{2k-1}])=0 \ ?\quad \in H_{2k-1}(BG^{\omega,\delta};\Z)
$$
holds or not.
\end{remark}
\vspace{2mm}
\noindent
{\it Proof of Theorem \ref{th:hatchi}}

First we prove $\mathrm{(i)}$.
The restriction of the central extension 
$0\rightarrow \Z\rightarrow \tilde{G}^\delta\rightarrow G^\delta\rightarrow 0$
to the subgroup $\mathrm{SO}(2)_{\text{tor}}\cong\Q/\Z\subset G^\delta$ is
$0\rightarrow \Z\rightarrow \Q\rightarrow \Q/\Z\rightarrow 0$.
The Gysin exact sequence of this central extension is given by
$$
H_{2k+1}(\Q;\Z)=0 \rightarrow H_{2k+1}(\Q/\Z;\Z)\cong\Q/\Z
\overset{\cap i_0^*\chi}{\rightarrow}
H_{2k-1}(\Q/\Z;\Z)\cong\Q/\Z \rightarrow H_{2k}(\Q;\Z)=0
$$
where $i_0: \mathrm{SO}(2)_{\text{tor}}\subset G^\delta$ denotes the inclusion.
Therefore the homomorphism $\cap i_0^*\chi$ is an isomorphism for all $k$.

Let
$0\rightarrow \Z\rightarrow \tilde{\Gamma}\rightarrow \Gamma\rightarrow 0$
be the restriction of the central extension to $\Gamma\subset G^\delta$. Then
we have the following commutative diagram between the Gysin exact
sequences
\begin{equation*}
\begin{CD}
0 @>>> H_{2k+1}(\Q/\Z;\Z)\cong\Q/\Z
 @>{\cap i_0^*\chi}>{\cong}>H_{2k-1}(\Q/\Z;\Z)\cong\Q/\Z  @>>> 0\\
 @VVV  @V{i_*}VV @V{i_*}VV @VVV\\
H_{2k+1}(\tilde{\Gamma};\Z) @>>> H_{2k+1}(\Gamma;\Z)
  @>{\cap i_\Gamma^*\chi}>> H_{2k-1}(\Gamma;\Z)@>>> H_{2k}(\tilde{\Gamma};\Z)
\end{CD}
\label{eq:tor}
\end{equation*}
where $i_\Gamma: \Gamma\subset G^\delta$ denotes the inclusion.
Hence, if the homomorphism $i_*$ on the righthand side is injective (resp. non-trivial),
so is the lefthand side as well. The claim follows from this.

Next we prove $\mathrm{(ii)}$.
By the assumption that $\chi^k=0\in H^{2k}(\Gamma;\Q)$, we can define the
secondary class $\widehat{\chi^k}\in H^{2k-1}(\Gamma;\Q/\Z)$ which is 
well defined modulo the image of the natural homomorphism
$H^{2k-1}(\Gamma;\Q)\rightarrow H^{2k-1}(\Gamma;\Q/\Z)$
(see Definition \ref{def:sc} for details).

On the other hand, $i_0^*\chi^k=0\in H^{2k}(\mathrm{SO}(2)_{\text{tor}};\Q)$ so that
we have also the secondary class $\widehat{i_0^*\chi^k}\in H^{2k-1}(\mathrm{SO}(2)_{\text{tor}};\Q/\Z)$.
By the naturality of the secondary class, this class is equal to 
$i^*\widehat{\chi^k}$. 
Furthermore, all the indeterminacy coming from the rational cohomology
vanishes here 
while the essential part remains so that this class is uniquely defined 
and it gives an
isomorphism
$$
\cap i^*\widehat{\chi^k}: H_{2k-1}(\mathrm{SO}(2)_{\text{tor}};\Z)\cong\Q/\Z
\overset{\cong}{\longrightarrow} \Q/\Z.
$$
Then the required claim follows because of the following identity
$$
\langle i_*(H_{2k-1}(\mathrm{SO}(2)_{\text{tor}};\Z)),\widehat{\chi^k}\rangle
=\langle H_{2k-1}(\mathrm{SO}(2)_{\text{tor}};\Z),i^*\widehat{\chi^k}\rangle
$$
together with the claim of $\mathrm{(i)}$. 
\qed

\begin{remark}
Here we give a sketch of proof of the following fact.
If $B\overline{\Gamma}_1$ were an Eilenberg MacLane space $K(\R,3)$, then we can
compute $H_{*}(B\mathrm{Diff}^\delta_+S^1;\Z)$. In particular,
it has no torsion. More precisely, we show the following.
Let 
$$
\mathrm{GV}: B\overline{\Gamma}_1\rightarrow K(\R,3)
$$
be the classifying map for the Godbillon-Vey class in $H^3(B\overline{\Gamma}_1;\R)$ and let
\begin{equation}
\wedge\mathrm{GV}: \wedge B\overline{\Gamma}_1\rightarrow \wedge K(\R,3)
\label{eq:vgv}
\end{equation}
be the associated map between free loop spaces induced by $\mathrm{GV}$. Now we see that
the extension class of the fibration 
$$
\Omega K(\R,3)=K(\R,2)\rightarrow \wedge K(\R,3)\rightarrow  K(\R,3)
$$
defined in $H^3(K(\R,3);\pi_2(K(\R,2)))\cong \mathrm{Hom}_\Z(\R,\R)$ is trivial
because this fibration has a section. Therefore $\wedge K(\R,3)$ is homotopy equivalent
to the product $K(\R,2)\times K(\R,3)$. If we put this into the map \eqref{eq:vgv}, then
we obtain a mapping
$$
\wedge\mathrm{GV}: \wedge B\overline{\Gamma}_1\rightarrow K(\R,2)\times K(\R,3).
$$
This corresponds to the 
two cohomology classes
$\alpha\in H^2(B\widetilde{\mathrm{Diff}}^\delta_+S^1;\R)$
(the Godbillon-Vey class integrated along the fibers) and
$\beta\in H^3(B\widetilde{\mathrm{Diff}}^\delta_+S^1;\R)$
(the Godbillon-Vey class) under the isomorphism 
$H^*(B\widetilde{\mathrm{Diff}}^\delta_+S^1;\R)\cong H^*(\wedge B\overline{\Gamma}_1;\R)$
induced by Thurston's Theorem \ref{th:Thurston}. 
Therefore the mapping \eqref{eq:vgv} induces a homomorphism
$$
(\wedge GV)_*: H_*(\wedge B\overline{\Gamma}_1;\Z)\rightarrow H_*(K(\R,2)\times K(\R,3);\Z)
\cong S^*_\Z(\R)\otimes_\Z \wedge^*_\Z(\R).
$$
Here in the last term, $S^k_\Z(\R)$ (resp. $\wedge^k_\Z(\R)$) denotes the $k$-th symmetric power over $\Z$
(resp. $k$-th exterior power over $\Z$) of $\R$ which is considered as a $\Q$-vector space.
We remark that the operation over $\Z$ is the same as that over $\Q$ because $\R$ is a
uniquely divisible group. 
Also the degree of the generator $S^1_\Z(\R)=\R$ of $S^{*}_\Z(\R)$ is $2$ while 
the degree of the generator $\wedge ^1_\Z(\R)=\R$ of $\wedge^{*}_\Z(\R)$ is $3$
(see \cite{MR877332} for more details of this computation).

Now we consider the Borel constructions on each space of the mapping \eqref{eq:vgv}. Then
we obtain a morphism of $S^1$-fibrations.
\begin{equation}
\begin{CD}
S^1 @=  S^1 @=  S^1\\
 @VVV  @VVV @VVV \\ 
  B\tilde{G}^\delta @>{H}>> \wedge B\overline{\Gamma}_1 @>>> \wedge K(\R,3)\\
   @VVV  @VVV @VVV \\ 
  B\tilde{G}^\delta /S^1=BG^\delta @>{H/\!/S^1}>> \wedge B\overline{\Gamma}_1/\!/S^1 @>>> \wedge K(\R,3)/\!/S^1.
\end{CD}
\label{eq:s1}
\end{equation}

This induces a homomorphism
\begin{equation}
H_*(BG^\delta;\Z)\overset{\text{Thurston}}{\cong} H^{S^1}_*(\wedge B\overline{\Gamma}_1;\Z)
\rightarrow H^{S^1}_*(\wedge K(\R,3);\Z)
\label{eq:bv}
\end{equation}
which would be an isomorphism if $B\overline{\Gamma}_1$ were $K(\R,3)$.
In general, it is an important and extremely difficult problem to determine
the kernel and cokernel of this homomorphism, both in the $C^\infty$ and 
real analytic categories.

Now we can determine $H^{S^1}_*(\wedge K(\R,3);\Z)$ as follows.
\begin{align*}
H^{S^1}_{2k}(\wedge K(\R,3);\Z)&\cong\Z\oplus Q_{2k}\\
H^{S^1}_{2k+1}(\wedge K(\R,3);\Z)&\cong Q_{2k+1}
\end{align*}
where each $Q_k$ is a certain $\Q$-vector space which can be described explicitly.
For example, the first several terms are given as follows.
\begin{align*}
Q_1&=0,\ Q_2=\R,\ Q_3=0,\ Q_4=S^2_\Z(\R),\ Q_5=\wedge^2_\Z(\R),\\
 Q_6&=S^3_\Z(\R),\
Q_7=S^{2,1}(\R),\ Q_8=S^4_\Z(\R)\oplus \wedge^3_\Z(\R),\
 Q_9=S^{3,1}(\R),\\
 Q_{10}&=S^5_\Z(\R)\oplus S^{2,1,1}(\R),\ Q_{11}=S^{4,1}(\R)
\oplus S^{2,1,1}(\R),\cdots
\end{align*}
Here the symbol $S^{k_1,k_2,\cdots}(\R)$ denotes the $\Q$-vector
space obtained by applying the Schur functor $S^{k_1,k_2,\cdots}\ (k_1\geq k_2\geq\cdots)$
on the $\Q$-vector space $\R$.

This is a purely homotopy theoretical computation using the Gysin sequence applied to the 
most right $S^1$-fibration of the diagram \eqref{eq:s1} together with the 
homology computation of the associated fibration
\[
\wedge K(\R,3)\rightarrow \wedge K(\R,3)/\!/S^1\rightarrow BS^1.
\]
However, it is easier to do this computation if we keep in mind
 geometric properties of the
characteristic classes $\chi$ and $\alpha, \beta$ as well as Schur functors in representation theory
(see \cite{MR1153249}). 
Here we omit the details (see \S 4 Appendix B for the first several cases of 
computations). 
\label{rem:kr3}
\end{remark}

\newpage
\section{Appendix: secondary classes}

Here we briefly describe definitions of secondary classes associated with 
vanishing of powers of the rational Euler class.
We consider the same situation as in the setting of Theorem \ref{th:hatchi}.
Thus let $\Gamma\subset \mathrm{Diff}^\delta_+S^1$ be any subgroup containing 
$\mathrm{SO}(2)$. Assume that $\chi^k=0\in H^{2k}(\Gamma;\Q)$.
Then we can define two secondary classes
\begin{align*}
\widehat{\chi^k}\in H^{2k-1}(\Gamma;\Q/\Z)\\
T\chi^{k}\in H^{2k-1}(\tilde{\Gamma};\Q)
\end{align*}
as follows.
First, choose
\begin{align*}
c&\in Z^2(B\mathrm{Diff}^\delta_+S^1;\Z)\quad \text{such that}\quad  \delta c=0\quad \text{and} \quad [c]=\chi
\\
b& \in C^1(B\widetilde{\mathrm{Diff}}^\delta_+S^1;\Z) \quad \text{such that}\quad  p^*c=\delta b .
\end{align*}
These choinces are
essentially unique (i.e. modulo exact cochains), because we can choose these cochains  at the levels
of $B\mathrm{Diff}_+S^1$ and $B\widetilde{\mathrm{Diff}}_+S^1$ and then 
$H^2(B\mathrm{Diff}_+S^1;\Z)$
is isomorphic to $\Z$ and $B\widetilde{\mathrm{Diff}}_+S^1$ is contractible.

\begin{definition}
$\mathrm{(i)}$\ 
By the assumption $\chi^k=0\in H^{2k}(\Gamma;\Q)$, there exists certain element $a\in C^{2k-1}(\Gamma;\Q)$
such that $c^k|_\Gamma=\delta a$. Set $\bar{a}\in C^{2k-1}(\Gamma;\Q/\Z)$ be the projection of $a$ under the coefficients
projection $\Q\rightarrow \Q/\Z$. Then we have $\delta\bar{a}=\bar{c}^k=0$. Define
$$
\widehat{\chi^k}=[\bar{a}]\in H^{2k-1}(\Gamma;\Q/\Z)
$$
which is well-defined modulo
$$
\mathrm{Image}[H^{2k-1}(\Gamma;\Q)\rightarrow H^{2k-1}(\Gamma;\Q/\Z)].
$$

$\mathrm{(ii)}$\ 
Under the same condition as above, we have
$$
\delta(p^* a -(b\ p^* c^{k-1})|_\Gamma)=p^* (c^k)|_\Gamma - (p^* c\cup  p^* c^{k-1})|_\Gamma)=0.
$$
Define
$$
T\chi^k=[p^* a -(b\ p^* c^{k-1})|_\Gamma]\in H^{2k-1}(\tilde\Gamma;\Q)
$$
which is well-defined modulo
$$
\mathrm{Image}[H^{2k-1}(\Gamma;\Q)\rightarrow H^{2k-1}(\tilde\Gamma;\Q)].
$$
\label{def:sc}
\end{definition}



\end{document}